\documentclass[12pt]{amsart}

\usepackage{amssymb}

\usepackage{enumitem}

\usepackage{graphicx}

\makeatletter
\@namedef{subjclassname@2020}{%
  \textup{2020} Mathematics Subject Classification}
\makeatother

\usepackage[T1]{fontenc}

\theoremstyle{definition}

\numberwithin{equation}{section}

\usepackage{amsthm}
\usepackage{amsmath}
\usepackage[mathscr]{eucal}
\usepackage{comment}

\usepackage{stmaryrd}

\usepackage{manfnt}

\AtBeginDocument{%
   \def\MR#1{}
}

\theoremstyle{plain}
\newtheorem{thm}{Theorem}[section]		
\newtheorem{prop}[thm]{Proposition}
\newtheorem{cor}[thm]{Corollary}
\newtheorem{lem}[thm]{Lemma}

\theoremstyle{definition}

\theoremstyle{remark}
\newtheorem{rmk}{Remark}[section]
\newtheorem*{ac}{Acknowledgements}

\newcommand{\zz}{\mathbb{Z}}

\newcommand{\rr}{\mathbb{R}}

\DeclareMathOperator{\card}{Card}

\DeclareMathOperator{\cl}{CL}

\DeclareMathOperator{\met}{Met}

\newcommand{\ult}[2]{\mathrm{UMet}(#1; #2)}

\DeclareMathOperator{\metdis}{\mathcal{D}}
\DeclareMathOperator{\umetdis}{\mathcal{UD}}

\DeclareMathOperator{\yodiam}{diam}

\newcommand{\yosub}{\subset}

\newcommand{\yotdim}[1]{\mathrm{dim}_{\mathrm{T}}(#1)}

\newcommand{\yorppp}[1]{\omega_{#1}}
\newcommand{\yormainmap}{\Phi}
\newcommand{\yorqmap}{\pi}

\newcommand{\yorind}[1]{\Theta(#1)}

\newcommand{\yorrho}[2]{\varrho_{#1, #2}}
\newcommand{\yorclosed}{F}

\newcommand{\yorqsp}[2]{#1_{/#2}}

\newcommand{\yorsubmap}{\Xi}
\newcommand{\yorsubmapq}{\Sigma}

\newcommand{\yorbd}[1]{\mathbf{#1}}

\newcommand{\yorop}[1]{\alpha #1}

\newcommand{\yorubdim}[1]{\overline{\dim}_{\mathrm{B}}(#1)}
\newcommand{\yorhdim}[1]{\dim_{\mathrm{H}}(#1)}

\newcommand{\yorpdim}[1]{\dim_{\mathrm{P}}(#1)}
\newcommand{\yoradim}[1]{\dim_{\mathrm{A}}(#1)}

\newcommand{\yordimset}[2]{\mathrm{FD}(#1; #2)}
\newcommand{\yordimsetq}[2]{\mathrm{TD}(#1; #2)}

\newcommand{\yordimsetu}[3]{\mathrm{UFD}(#1; #2; #3)}
\newcommand{\yordimsetqu}[3]{\mathrm{UTD}(#1; #2; #3)}

\newcommand{\yorlset}{\mathcal{A}}
\newcommand{\yorrset}{\mathcal{B}}

\newcommand{\yorzeron}{\yorbd{n}_{0}}
\newcommand{\yorzeroz}{\yorbd{z}_{0}}

\newcommand{\stst}[2]{#1[#2]}

\DeclareMathOperator{\ubdim}{\overline{\dim}_{B}}
\DeclareMathOperator{\hdim}{\dim_{H}}

\DeclareMathOperator{\pdim}{\dim_{P}}
\DeclareMathOperator{\adim}{\dim_{A}}

\newcommand{\mnum}{\mathrm{N}}

\newcommand{\yorcompset}{\mathrm{Cm}}
\newcommand{\yorcompsetu}{\mathrm{UCM}}
\newcommand{\yorpropset}{\mathrm{Pr}}
\newcommand{\yorpropsetu}{\mathrm{UPr}}

\newcommand{\yorflatmap}{\yormainmap}
\newcommand{\yorflatmapx}{r}
\newcommand{\yorflatmapy}{\yorqmap}

\newcommand{\yorfmet}{\lambda}
\newcommand{\yorvmet}{v}

\newcommand{\yorsetballb}{\widetilde{B}}
\newcommand{\yorsetballu}{\widetilde{U}}

\newcommand{\yorex}{\widetilde{E}}
\makeatletter
\@addtoreset{equation}{section}

\makeatother

\begin{document}

\title[Factorization]
{
A factorization of 
metric spaces
}
\author[Yoshito Ishiki]
{Yoshito Ishiki}

\address[Yoshito Ishiki]
{\endgraf
Photonics Control Technology Team
\endgraf
RIKEN Center for Advanced Photonics
\endgraf
2-1 Hirasawa, Wako, Saitama 351-0198, Japan}

\email{yoshito.ishiki@riken.jp}

\date{\today}
\subjclass[2020]{Primary 54E35, 
Secondary 
51F99}
\keywords{Retraction, Zero-dimensional spaces, Embedding}

\begin{abstract}
We first prove that 
for every metrizable space $X$, 
for every  closed 
subset $F$ whose complement is zero-dimensional, 
the space $X$ can be 
embedded
as a closed subset  into 
 a  product of the closed subset $F$ and 
a metrizable zero-dimensional space. 
Using this theorem, we next  show 
the existence of extensors 
 of metrics and ultrametrics, which 
 preserve properties of metrics  such as the completeness, the properness, 
being an ultrametrics, 
 its fractal dimensions, and 
 large scale structures. 
This result contains
some of  the author's extension 
theorems of ultrametrics. 
\end{abstract}

\maketitle

\section{Introduction}\label{sec:intro}

Under certain assumptions, 
for a topological space $X$, and for a retract $E$ of $X$, 
Michael \cite{MR1909003} 
established the existence of a perfect map 
$f$ (or a closed map)
 from 
$X$ into the 
 product space $E\times \rr$ such that 
$f(a)=(a, 0)$ for all $a\in E$. 
Michael also  
characterized the property 
 that a retract $E$ becomes 
a
perfect (or  closed) retract  of $X$ using 
the behavior of $E\times \{0\}$ in  a
product space $E\times \rr$. 

As an improvement of Michael's results, 
in this paper,  
we shall prove that 
for every metrizable space $X$, for every 
closed 
subset $F$ whose complement is zero-dimensional
(i.e., it has covering dimension $0$),  
there exists a  topological embedding 
$\yormainmap$ of the 
whole space $X$ into the  product space
$F\times \Lambda$
 of the closed subset $F$ and 
a metrizable zero-dimensional space $\Lambda$
(see Theorem \ref{thm:main1}), where the factor 
$\Lambda$ is a metric quotient space identifying  
$F$ as a single point. 
Moreover, 
the embedding 
 $\yormainmap$ is a closed map  and 
there exists $\theta\in \Lambda$ such that 
$\yormainmap(a)=(a, \theta)$ for all $a\in F$. 
Based on this theorem, we can roughly 
consider that 
the space $X$ is  factorized into a product of  $F$ and some 
$0$-dimensional metrizable space,  and 
we can take a retraction from $X$ into $F$ as 
a projection in a product space in some sense. 
 
To prove  Theorem \ref{thm:main1}, 
we utilize the 
Engelking's  theorem  \cite{MR239571} on the existence of retraction, stating  that
if $X$ is a metrizable space and 
$F$ is a non-empty closed subset
 for which 
$X\setminus F$ is $0$-dimensional, 
then there exists a retraction 
$f\colon X\to F$, which is a closed map. 
The construction of this  retraction 
includes stronger properties than being a closed map, 
which we will use  in the present  paper. 

Before explaining applications of our embedding  theorem, 
we introduce some notions on spaces of metrics and 
ultrametrics. 
A metric $d$ on a set $X$ is said to be an \emph{ultrametric} if 
it satisfies the so-called strong triangle inequality 
$d(x, y)\le d(x, z)\lor d(z, y)$ for all $x, y, z\in X$, 
where the symbol ``$\lor$'' stands for the maximum operator on 
$\rr$. 
Let $S$ be a subset of 
$[0, \infty)$ with $0\in S$, 
and $X$ be a metrizable space. 
We denote by $\met(X; S)$
(resp.~$\ult{X}{S}$) the set of all 
metrics (resp.~ultrametrics) on $X$ generating the same topology of $X$
and 
taking values in $S$. 
We simply write $\met(X):=\met(X; [0, \infty))$. 
Let $\metdis_{X}$ be the supremum distance on 
$\met(X)$. Namely, 
$\metdis_{X}(d, e)=\sup_{x, y\in X}|d(x, y)-e(x, y)|$. 
Then $\metdis_{X}$ is a metric on $\met(X)$
taking values in $[0, \infty]$. 
Using open balls,  as is the case of  ordinary metric spaces, we can introduce a topology 
on $\met(X)$ with $\metdis_{X}$.
In what follows, we always consider that 
$\met(X)$  is equipped with the 
topology generated by $\metdis_{X}$. 
Similarly, for $d, e\in \ult{X}{S}$, 
we define $\umetdis_{X}^{S}(d, s)$ by the 
infimum $\epsilon \in S\setminus \{0\}$ such that 
$d(x, y)\le e(x, y)\lor \epsilon$ and 
$e(x, y)\le d(x, y)\lor \epsilon$ for all 
$x, y\in X$. 
Then $\umetdis_{X}^{S}$ is an ultrametric on 
$\ult{X}{S}$
taking values in $\cl(S)\sqcup\{\infty\}$. We consider that $\ult{X}{S}$ is equipped with the topology generated by $\ult{X}{S}$.

In \cite{Ishiki2021ultra}, 
the author showed that for
an ultrametrizable space $X$, 
for a closed subset $F$, 
for a subset $S$ of $[0, \infty)$ with $0\in S$, 
and 
for every $d\in \ult{F}{S}$, 
there exists $D\in \ult{X}{S}$ such that 
$D|_{F^{2}}=d$. 
This result  is an ultrametric analogue of 
Hausdorff's metric extension theorem 
\cite{Ha1930}. 
Some authors have investigated 
variants of the Hausdorff's metric extension 
theorem 
(see \cite{NN1981}, \cite{MR3135687},  and \cite{MR3090172}).

In \cite{ishiki2022highpower}, 
the author generalized the results \cite{NN1981}
on the existence of simultaneous  extensions of metrics
 to the theory of metrics 
taking values in general linearly ordered Abelian groups. 
This theorem also contains the existence of 
simultaneous extensions of ultrametrics. 
Namely,  for every ultrametrizable space 
$X$, a closed subset $F$, and a subset $S$ of $[0, \infty)$ with $0\in S$, 
the author obtained an 
isometric extensor 
$\Upsilon\colon \ult{F}{S}\to \ult{X}{S}$. 
For extensors of (ultra)metrics, 
we refer the readers to  \cite{tymchatyn2005note},
\cite{stasyuk2009continuous} and 
\cite{NN1981}. 

A metric $d$ on a set $X$ 
(or a metric space $(X, d)$) is said to be \emph{proper} if 
all closed balls of $(X, d)$ are compact. 
In \cite{Ishiki2022proper}, 
the author proved that a proper metric 
$d\in \met(F)$ can be extended to 
a proper metric $D\in\met(X)$, 
and also proved an  analogue of this  extension theorem  for proper ultrametrics. 

In the present paper, 
using our theorem on embedding 
a metrizable space into a product space 
(Theorem \ref{thm:main1}), 
we will  unify 
the author's extension theorems on 
ultrametrics explained above.
Namely, using a structure of a product space, 
under the same assumptions to Theorem \ref{thm:main1}, 
we obtain extensors 
$\yorsubmap, \yorsubmapq\colon \met(F)\to \met(X)$  which preserve  properties of 
$d\in \met(F)$  such as the completeness, the properness, 
being an  ultrametric, 
its fractal dimensions, 
and large scale structures
(see Theorems \ref{thm:subARC} and 
\ref{thm:subNARC}).

The organization of this paper is as follows: 
In Section \ref{sec:pre}, 
we prepare some notions on metric spaces. 
We also investigate basic properties of 
metric quotient spaces $\yorqsp{X}{F}$
and 
we review the construction of Engelking's 
retraction in \cite{MR239571}. 
Section \ref{sec:proof} is devoted to 
the proof of  Theorem \ref{thm:main1}, which is our first main result. 
In Section \ref{sec:app}, 
we first explain the definitions  of the fractal dimensions. 
As applications of Theorem \ref{thm:main1}, 
we shall prove our remaining main results,  Theorems \ref{thm:subARC} and 
\ref{thm:subNARC}.

\section{Preliminaries}\label{sec:pre}
In this section, we prepare some notions and notations. 
\subsection{Generalities}
For a metric space $(X, d)$, 
and for $x\in X$ and 
$\epsilon\in (0, \infty)$, 
we denote by $B(x, \epsilon; d)$ 
(resp.~$U(x, \epsilon; d)$)
the closed (resp.~open) ball of $(X, d)$ centered 
at $x$ with radius $\epsilon$. 
We  simply represent  it as
$B(x, \epsilon)$ 
(resp.~$U(x, \epsilon)$) when no confusion can arise. 
For a subset $A$ of $X$, we define 
$\yorrho{d}{A}\colon X\to [0, \infty)$ by 
$\yorrho{d}{A}(x)=\inf_{y\in A}d(x, y)$. 
Notice that $\yorrho{d}{A}$ is continuous. 
Moreover, it is $1$-Lipschitz. 

A map $f\colon X\to Y$ between topological spaces 
$X$ and $Y$ is said to be 
\emph{open} (resp.~\emph{closed}) if 
for every open (resp.~closed) subset $A$ of $X$, 
the image  $f(A)$ is open (resp.~closed). 

Let $(X, d)$ and $(Y, e)$ 
be metric spaces. 
We define product metrics 
$d\times_{1}e$ and 
$d\times_{\infty}e$
 by 
\[
(d\times_{1}e)((x, y), (u, v))=d(x, u)+e(y, v)
\]
and 
\[
d\times_{\infty}e((x, y), (u, v))=
d(x, u)\lor e(y, v).
\] 
Notice that $d\times_{1}e$ and 
$d\times_{\infty}e$ generate the same topology 
of $X\times Y$. 
Remark that  if $e$ and $d$ are ultrametrics, 
then so is $d\times_{\infty}e$. 
We can also define the $\ell^{p}$-product metric 
$d\times_{p}e$ for all $p\in (1, \infty)$. 
However, we omit it since we only consider the 
$\ell^{1}$ and $\ell^{\infty}$-product metrics in 
this paper. 

\begin{lem}\label{lem:prodproper}
Let $(X, d)$ and $(Y, e)$ be metric spaces. 
If both $d$ and $e$ are proper, 
then so are $d\times_{1} e$ and $d\times_{\infty} e$. 
\end{lem}
\begin{proof}
For all 
$\epsilon\in (0, \infty)$, 
and for every point 
$p=(a, b)\in X\times Y$, 
the sets
$B(p, \epsilon; d\times_{1}e)$ and 
$B(p, \epsilon; d\times_{\infty}e)$ are
contained in 
$B(a, \epsilon; d)\times B(b, \epsilon; e)$. 
Since $d$ and $e$ are proper, 
the set $B(a, \epsilon; d)\times B(b, \epsilon; e)$ is 
compact. 
Then
$B(p, \epsilon; d\times_{1}e)$ and 
$B(p, \epsilon; d\times_{\infty}e)$ are compact. 
Hence 
 $d\times_{1}e$ and $d\times_{\infty}e$ are 
proper. 
\end{proof}

\subsection{Topological dimension}

A subset of a topological space is said to be 
\emph{clopen} if it is closed and open in the space. 
For a topological space $X$, 
we write $\yotdim{X}=0$ if 
$X\neq \emptyset$  and 
every finite covering of $X$ has 
a refinement covering consisting of 
finitely many mutually disjoint open subsets of $X$. 
A topological space 
$X$ is said to be \emph{ultranormal} if 
for each pair of  disjoint closed subsets $A$ and $B$ of 
$X$, there exists a clopen subset $C$ of $X$ such that 
$A\yosub C$ and $B\cap C=\emptyset$. 
Remark that a topological space is ultranormal 
if and only if it has large inductive dimension
(usually denoted by $\mathrm{Ind}$) $0$. 
 For more discussions on topological dimensions, 
 we refer the readers to \cite{Cdimension}
 and 
 \cite{MR0394604}. 
 
The proof of the next proposition is 
presented in  \cite[Proposition 2.9]{Cdimension}
and \cite[Proposition 2.3]{MR0394604}. 
\begin{prop}\label{prop:zeroultra}
A
non-empty topological space $X$ is 
 ultranormal if and only if $\yotdim{X}=0$. 
\end{prop}

\begin{rmk}
In what follows, 
without referring to Proposition \ref{prop:zeroultra}, 
we use the equivalence between 
$\yotdim{X}=0$ and the ultranormality. 
\end{rmk}

A subset $S$ of $[0, \infty)$ is 
said to be 
\emph{characteristic} if 
$0\in S$ and for all $t\in (0, \infty)$, 
there exists $s\in S\setminus \{0\}$ such that 
$s\le t$. 
The next proposition is deduced from 
\cite[Proposition 2.14]{Ishiki2021ultra} and 
\cite[Theorem II]{MR80905}. 

\begin{prop}\label{prop:S-ult}
For every  characteristic subset $S$ of $[0, \infty)$, 
a metrizable space $X$ satisfies 
$\ult{X}{S}\neq \emptyset$ if and only if 
$\yotdim{X}=0$. 
\end{prop}

For a topological space $T$, 
and for a subset $E$ of $T$, 
we denote by $\cl(E)$ (resp.~$\partial E$)
the closure of $E$ (resp.~the boundary of $E$) in 
$T$. 
\begin{lem}\label{lem:clopen}
Let $X$ be a metrizable space, 
and $A$ be a 
non-empty closed subset of $X$ such that 
$\yotdim{X\setminus A}= 0$. 
If $U$ is  an open subset of $X$ with
$A\yosub U$, 
then 
 there exists a clopen subset $V$ of $X$ 
 such that $A\yosub V$ and $V\yosub U$. 
\end{lem}
\begin{proof}
Since $X$ is normal, 
there exists an open subset $G$ of $X$ with 
$A\yosub G$ and $\cl(G)\yosub U$. 
Since $\partial G\yosub X\setminus A$
and $\yotdim{X\setminus A}=0$, 
we can find a clopen subset 
$Q$ of $X\setminus A$ such that 
$\partial G\yosub Q$ and 
$Q\yosub U\cap (X\setminus A)$. 
Put $V=\cl(G)\cup Q$. 
Then  $V$ is a clopen  set as required. 
\end{proof}

We say that a topological space is a \emph{Cantor space} if it is homeomorphic to the (middle-third)
Cantor set. Note that all Cantor spaces are 
homeomorphic to each other. 
\begin{prop}\label{prop:sepembzero}
If $X$ is a separable metrizable space with 
$\yotdim{X}=0$,
then $X$ can be topologically embedded into 
a Cantor space. 
\end{prop}
\begin{proof}
The proposition is well-known (see  for example 
\cite[Exercise 9.14]{Cdimension}). 
We briefly review a construction of an embedding. 
Let $\{U_{i}\}_{i\in \zz_{\ge 0}}$ be a  base of $X$
consisting of clopen subsets of $X$. Let $f_{i}\colon X\to \{0, 1\}$ be the  characteristic map of $U_{i}$, i.e., 
$f_{i}(x)=1$ if and only if $x\in U_{i}$. 
We define $g\colon X\to \{0, 1\}^{\zz_{\ge 0}}$ by 
$g(x)=(f_{i}(x))_{i\in \zz_{\ge 0}}$. 
Then $g$ is a topological embedding, and 
$\{0, 1\}^{\zz_{\ge 0}}$ is a Cantor space. 
\end{proof}

\subsection{Metric quotient  spaces}
In this section, we construct a factor of 
a target  product space in our first main theorem 
(Theorem \ref{thm:main1}).
That factor  is a metric quotient space 
identifying a given closed subset as a single point. 
To consider metric quotient space, we use 
Hausdorff's  construction  
\cite{Ha1930} of a metric vanishing on a given 
closed subset.
Let $X$ be a metrizable space, 
and 
$F$ be a  non-empty closed subset of $X$. 
Take a point $\yorppp{F}$ satisfying that 
$\yorppp{F}\not \in X\setminus F$. 
We put $\yorqsp{X}{F}=(X\setminus F)\sqcup \{\yorppp{F}\}$. 
For the sake of convenience, 
we choose  $\yorppp{F}$ as a point of $F$. 
For a metric $d\in \met(X)$, 
we define   a symmetric function 
$\yorind{d}\colon (\yorqsp{X}{F})^{2}\to [0, \infty)$ by 
$\yorind{d}(x, y)=\min\{d(x, y), \yorrho{d}{F}(x)+\yorrho{d}{F}(y)\}$. 
Notice that $\yorind{d}(x, \yorppp{F})=\yorrho{d}{F}(x)$
for all $x\in \yorqsp{X}{F}$
since $\yorppp{F}\in F$. 
In Proposition 
 \ref{prop:inducedmetric}, 
we will prove that 
$\yorind{d}$ is actually  a metric on 
$\yorqsp{X}{F}$. 
We define a map 
$\yorqmap\colon X\to \yorqsp{X}{F}$ by 
$\yorqmap(x)=x$ if 
$x\in X\setminus F$; otherwise, 
$\yorqmap(x)=\yorppp{F}$. 
Remark that the metric structure of $(\yorqsp{X}{F}, \yorind{d})$ does 
not depend on the choice of $\yorppp{F}$.

\begin{prop}\label{prop:inducedmetric}
Let $X$ be a metrizable space,  and 
$F$ be a 
non-empty closed subset of $X$. 
If $d\in \met(X)$, 
then the following statements are true:
\begin{enumerate}[label=\textup{(\arabic*)}]
\item\label{item:theta1}
The function $\yorind{d}$ is a metric on $\yorqsp{X}{F}$. 
\item\label{item:theta2}
 The map $\yorqmap \colon X\to \yorqsp{X}{F}$ is
 $1$-Lipschitz. In particular, it is 
continuous. 
\item\label{item:theta3}
 The restricted map 
$\yorqmap|_{X\setminus F}\colon (X\setminus  F)\to (\yorqsp{X}{F}\setminus \{\yorppp{F}\})$ is a  homeomorphism. 
\end{enumerate}
\end{prop}
\begin{proof}
The proof of \ref{item:theta1} is 
essentially the same as in  \cite{Ha1930}. 
To prove that $\yorind{d}$ is a metric, 
it suffices to show that $\yorind{d}$ satisfies the 
triangle inequality. 
Recalling the definition of $\yorrho{d}{F}(x)$, 
we notice that for all $x, y, z\in \yorqsp{X}{F}$, 
all of the following inequalities are true:
\begin{enumerate}[label=\textup{(\roman*)}]
\item $d(x, y)\le d(x, z)+d(z, y)$;
\item $\yorrho{d}{F}(x)+\yorrho{d}{F}(y)\le (\yorrho{d}{F}(x)+\yorrho{d}{F}(z))+d(z, y)$;
\item $\yorrho{d}{F}(x)+\yorrho{d}{F}(y)
\le d(x, z)+(\yorrho{d}{F}(z)+\yorrho{d}{F}(y))$;
\item $\yorrho{d}{F}(x)+\yorrho{d}{F}(y)
\le (\yorrho{d}{F}(x)+\yorrho{d}{F}(z))
+(\yorrho{d}{F}(z)+\yorrho{d}{F}(y))$. 
\end{enumerate}
Then we conclude that 
$\yorind{d}$ satisfies the triangle inequality.

We now verify  \ref{item:theta2}. 
For all $x, y\in X$, 
we have 
\[
\yorind{d}(\yorqmap(x), \yorqmap(y))
=\min\{d(x, y), \yorrho{d}{F}(x)+\yorrho{d}{F}(y)\}.
\] 
Thus we obtain $\yorind{d}(\yorqmap(x), \yorqmap(y))
\le d(x, y)$. 
This means that $\yorqmap$ is $1$-Lipschitz.

We next show that 
$\yorqmap|_{X\setminus F}\colon 
(X\setminus F)\to 
(\yorqsp{X}{F}\setminus \{\yorppp{F}\})$ is a  homeomorphism. 
Remark that 
$\yorqsp{X}{F}\setminus \{\yorppp{F}\}=X\setminus F$
as a set. 
For all $x\in X\setminus F$, 
 every number  $\epsilon \in (0, \yorrho{d}{F}(x))$ satisfies 
that 
$U(x, \epsilon; d)=U(\yorqmap(x), \epsilon; \yorind{d})$. 
Thus, the restricted map 
$\yorqmap|_{X\setminus F}$ is a homeomorphism. 
\end{proof}

\begin{rmk}Let $(X, d)$ be a metric space, 
and 
$F$ be a closed subset of $X$. 
It 
is not always true  that 
$\yorqmap\colon X\to \yorqsp{X}{F}$ is 
 open or closed. 
 We give an example. 
Put $X=\rr^{2}$ and 
$F=\rr\times \{0\}$. 
Define a metric on $d$ on $X$ $d(x, y)=|x_{1}-y_{1}|+|x_{2}-y_{2}|$, where 
$x=(x_{1}, x_{2})$ and $y=(y_{1}, y_{2})$. 
Let $A=\{\, (x_{1}, x_{2})\in X\mid x_{2}=(|x_{1}|+1)^{-1}\, \}$. 
Then $A$ is closed in $X$, and 
$\yorqmap(A)$ is not closed in $\yorqsp{X}{F}$ 
since $\yorppp{F}\not \in \yorqmap(A)$ and  
$U(\yorppp{F}, \epsilon)\cap \yorqmap(A)\neq \emptyset$ for all $\epsilon\in (0, \infty)$. 
Similarly, the set 
$O=
\{\, (x_{1}, x_{2})\in X\mid x_{2}<(|x_{1}|+1)^{-1}\, \}$ is 
 open in $X$ and $\yorqmap(O)$ is not open in 
 $\yorqsp{X}{F}$. 
 Namely, in this case, the map $\yorqmap$ is 
 neither open nor closed. 
\end{rmk}

The next lemma follows
from the fact that a continuous image of 
a separable (resp.~compact) space is 
separable (resp.~compact). 
\begin{prop}\label{prop:sepcpt}
Let $X$ be a metrizable space, 
and $F$ be a non-empty closed subset of $X$. 
If $d\in \met(X)$, 
then the following statements  are true: 
\begin{enumerate}[label=\textup{(\arabic*)}]
\item\label{item:26:sep}
 If $(X, d)$ is separable, then so is $(\yorqsp{X}{F}, \yorind{d})$. 
\item\label{item:26:cpt}
If $(X, d)$ is compact, then so is $(\yorqsp{X}{F}, \yorind{d})$. 
\end{enumerate}
\end{prop}

Let $(X, d)$ be a metric space. 
For a subset $A$ of $X$, 
and for $\epsilon\in (0, \infty)$, 
we define  
$\yorsetballb(A, \epsilon; d)=
\{\, x\in X\mid \yorrho{d}{A}(x)\le \epsilon\, \}$
and  
$\yorsetballu(A, r; d)=\{\, x\in X\mid \yorrho{d}{A}(x)<\epsilon\, \}$. 
Remark that 
$\yorsetballb(A, \epsilon; d)$ 
is closed  and 
$\yorsetballu(A, \epsilon; d)$ is open in $(X, d)$
since $\yorrho{d}{A}$ is continuous on $(X, d)$. 
\begin{prop}\label{prop:inducedzero}
Let $X$ be a metrizable space, 
and 
$F$ be a non-empty closed subset of $X$. 
If $d\in \met(X)$ and  $\yotdim{X\setminus F}=0$, then 
$\yotdim{\yorqsp{X}{F}}=0$ with respect to 
the topology generated by $\yorind{d}$. 
\end{prop}
\begin{proof}
Notice that for all $\eta\in (0, \infty)$, we have 
$\yorqmap(\yorsetballb(F, \eta ; d))=B(\yorppp{F}, \eta; \yorind{d})$ and $\yorqmap(\yorsetballu(F, \eta ; d))=U(\yorppp{F}, \eta; \yorind{d})$. 
First we show that the point $\yorppp{F}$
has a neighborhood system consisting of 
clopen subsets of $\yorqsp{X}{F}$. 
Take an arbitrary number 
$\epsilon\in (0, \infty)$. 
Put $G=U(\yorppp{F}, \epsilon; \yorind{d})$.
In this case, we have 
$\yorqmap^{-1}(G)
=\yorsetballu(F, \epsilon; d)$. 
Applying Lemma \ref{lem:clopen}
to $\yorsetballb(F, \epsilon/2; d)$ and 
$\yorqmap^{-1}(G)(=\yorsetballu(F; \epsilon; d))$, 
we can find a 
clopen subset $H$ of $X$ such that 
$\yorsetballb(F, \epsilon/2; d)\yosub H$ and 
$H\yosub \yorqmap^{-1}(G)$. 
Put $P=\yorqmap(H)$. 
Then we obtain 
$B(\yorppp{F}, \epsilon/2; \yorind{d})\yosub P$ and 
$P\yosub G$. 
We now prove $P$ is a clopen neighborhood of 
$\yorppp{F}$ in $\yorqsp{X}{F}$ with $P\yosub G$. 
Due to  $B(\yorppp{F}, \epsilon/2; \yorind{d})\yosub P$, 
the set $P$ is a 
neighborhood of $\yorppp{F}$. 
Since $\yorqmap|_{X\setminus F}\colon 
X\setminus F\to \yorqsp{X}{F}\setminus \{\yorppp{F}\}$  is a homeomorphism
(see \ref{item:theta3} in Proposition \ref{prop:inducedmetric})
 and 
$\yorqsp{X}{F}\setminus \{\yorppp{F}\}$ is 
open in $\yorqsp{X}{F}$, 
we notice that 
$\yorqmap(X\setminus H)$ 
and  $\yorqmap(H\setminus \yorsetballb(F, \epsilon/4; d))$
are open in 
$\yorqsp{X}{F}$.
Remark that we obtain  the equalities 
\begin{align}
&\yorqmap(X\setminus H)=\yorqsp{X}{F}\setminus P, \label{al:2222}\\
&\yorqmap(H\setminus \yorsetballb(F, \epsilon/4; d))=P\setminus B(\yorppp{F}, \epsilon/4; \yorind{d}).\label{al:2121}
\end{align}
Then 
the sets 
$\yorqsp{X}{F}\setminus P$
and $P\setminus B(\yorppp{F}, \epsilon/4; \yorind{d})$ 
are open in $\yorqsp{X}{F}$. 
In particular, 
the set $P$ is closed in $\yorqsp{X}{F}$. 
We also obtain 
$P=U(\yorppp{F}, \epsilon/2; \yorind{d})\cup 
(P\setminus B(\yorppp{F}, \epsilon/4; \yorind{d}))$, 
and hence 
 $P$  is open in $\yorqsp{X}{F}$. 
Therefore 
the set  $P$  is clopen in $\yorqsp{X}{F}$ such that 
$\yorppp{F}\in P$ and $P\yosub G$. 
This implies  that $\yorppp{F}$ has a clopen neighborhood system. 

To prove $\yotdim{\yorqsp{X}{F}}=0$, 
we verify that $\yorqsp{X}{F}$ is ultranormal. 
Let $A$ and $B$ be disjoint  closed subsets of $\yorqsp{X}{F}$. 
We may assume that $\yorppp{F}\not \in B$ and 
$B\neq \emptyset$. 
Due to the argument  discussed above, 
 there exists a clopen neighborhood $V$ of 
$\yorppp{F}$ with $V\cap B=\emptyset$. 
Since $\yorqsp{X}{F}\setminus V=(X\setminus F)
\setminus (V\setminus \{\yorppp{F}\})$ as a set, 
the statement \ref{item:theta3} in 
Proposition  \ref{prop:inducedmetric} 
yields  
$\yotdim{\yorqsp{X}{F}\setminus V}=0$. 
Thus 
we can take  a clopen subset $Q$ of $X\setminus V$
such that $A\setminus V\yosub Q$ and $Q\cap B=\emptyset$ 
(it can happen that $A\setminus V=A$). 
Put $O=Q\cup V$. Then $O$ is clopen in 
$\yorqsp{X}{F}$ and $A\yosub O$ and $O\cap B=\emptyset$. 
This means that 
$\yorqsp{X}{F}$ is ultranormal, and hence
$\yotdim{\yorqsp{X}{F}}=0$. 
\end{proof}

For a set $S$, let  $\card(S)$ denote the 
cardinality of $S$. 
\begin{prop}\label{prop:inducedcomplete}
Let $X$ be a metrizable space, 
 and 
$F$ be a non-empty
 closed subset of $X$. 
If $d\in \met(X)$ is complete, 
then so is $\yorind{d}$. 
\end{prop}
\begin{proof}
Take a Cauchy sequence $\{p_{i}\}_{i\in \zz_{\ge 0}}$ of 
$(\yorqsp{X}{F}, \yorind{d})$. 
We only need to show that there exists a convergent  subsequence of 
$\{p_{i}\}_{i\in \zz_{\ge 0}}$. 
Put $p_{i}=\yorqmap(x_{i})$ for each $i\in \zz_{\ge 0}$. 
For each $n\in \zz_{\ge 0}$, 
we define 
$W_{n}=\{\, i\in \zz_{\ge 0}\mid \yorrho{d}{F}(x_{i})\le 2^{-n}\, \}$. 
We divide the proof into two parts. 

Case 1.~(For all $n\in \zz_{\ge 0}$, we have 
$\card(W_{n})=\aleph_{0}$): 
In this case,  
we can take a strictly increasing function 
$\phi\colon \zz_{\ge 0}\to \zz_{\ge 0}$ such that 
$\phi(n)\in W_{n}$ for all $n\in \zz_{\ge 0}$. 
Then $\yorrho{d}{F}(x_{\phi(n)})\le 2^{-n}$ for all $n\in \zz_{\ge 0}$. Thus we also obtain 
$\lim_{i\to \infty}\yorind{d}(p_{\phi(i)}, \yorppp{F})=\lim_{i\to \infty}\yorrho{d}{F}(x_{\phi(i)})\to 0$, 
and hence 
$p_{\phi(i)}\to \yorppp{F}$ as $i\to \infty$. 
Therefore $\{p_{i}\}_{i\in \zz_{\ge 0}}$ has a convergent subsequence $\{p_{\phi(i)}\}_{i\in \zz_{\ge 0}}$. 

Case 2.~(There exists $k\in \zz_{\ge 0}$ such that 
$\card(W_{k})<\aleph_{0}$): 
In this case, we have 
$\card(\zz_{\ge 0}\setminus W_{k})=\aleph_{0}$. 
Thus, we can take a strictly increasing function 
$\phi\colon \zz_{\ge 0}\to \zz_{\ge 0}$ such that 
$\phi(i)\in \zz_{\ge 0}\setminus W_{k}$ for all 
$i\in \zz_{\ge 0}$. 
Then we have $2^{-k}<\yorrho{d}{F}(x_{\phi(i)})$ for all 
$i\in \zz_{\ge 0}$. 
Since $\{p_{\phi(i)}\}_{i\in \zz_{\ge 0}}$ is Cauchy, 
we notice that 
$\yorind{d}(p_{\phi(n)},p_{\phi(m)})\to 0$ as 
$n, m\to \infty$. 
From $2^{-k}<\yorrho{d}{F}(x_{\phi(i)})$ for all 
$i\in \zz_{\ge 0}$ and from
the definition of $\yorind{d}$, it follows that 
for all sufficiently large numbers $n, m\in \zz_{\ge 0}$, 
 we have 
$\yorind{d}(p_{\phi(n)}, p_{\phi(m)})=d(x_{\phi(n)}, x_{\phi(m)})$ and 
$d(x_{\phi(n)}, x_{\phi(m)})\to 0$ as $m, n\to \infty$. 
Namely,  the sequence $\{x_{\phi(i)}\}_{i\in \zz_{\ge 0}}$ is 
Cauchy in $(X, d)$. 
Since $d$ is complete, there exists a 
limit of $\{x_{\phi(i)}\}_{i\in \zz_{\ge 0}}$, 
say $l\in X$. 
Then $\{p_{\phi(i)}\}_{i\in \zz_{\ge 0}}$ converges to 
$\yorqmap(l)\in \yorqsp{X}{F}$. 
This finishes the proof. 
\end{proof}

For a non-compact,  locally compact,  
Hausdorff space 
$X$, we denote by $\yorop{X}$ the 
one-point compactification of $X$. 
Note that $\yorop{X}=X\sqcup\{\infty\}$ as a 
set. 
Remark that for  
a non-compact closed subset $F$ of $X$, 
 the space $\yorop{F}$ can be 
regarded as a closed subset of $\yorop{X}$ in a canonical way.

\begin{prop}\label{prop:qlocpt}
Let $X$ be a non-compact,  $\sigma$-compact, 
locally compact, 
 metrizable space, 
and $F$ be a non-compact closed subset of 
$X$. 
If $d\in \met(X)$ and $D\in \met(\yorop{X})$ 
satisfy that
$D|_{X^{2}}=d$, 
then 
$(\yorqsp{\yorop{X}}{\yorop{F}}, \yorind{D})$
is 
isometric 
 to 
$(\yorqsp{X}{F}, \yorind{d})$. 
In particular, the space 
$(\yorqsp{X}{F}, \yorind{d})$
 is compact. 
\end{prop}
\begin{proof}
We can choose $\yorppp{\yorop{F}}\in \yorqsp{\yorop{X}}{\yorop{F}}$ so that 
$\yorppp{\yorop{F}}=\yorppp{F}$. 
Then we obtain 
$\yorqsp{\yorop{X}}{\yorop{F}}=\yorqsp{X}{F}$ as a set. 
Since $X\setminus F=\yorop{X}\setminus \yorop{F}$, 
we conclude that 
$\yorind{D}(x, y)=\yorind{d}(x, y)$ for all 
$x, y\in \yorqsp{X}{F}$. 
Thus the spaces  $(\yorqsp{X}{F}, \yorind{d})$ and 
$(\yorqsp{\yorop{X}}{\yorop{F}}, \yorind{D})$ are isometric. 
Therefore the statement \ref{item:26:cpt} in Proposition \ref{prop:sepcpt}
completes the proof. 
\end{proof}

\begin{rmk}
There exist a metrizable space $X$, a non-empty subset $F$ of 
$X$, and  metrics 
$d, e\in \met(X)$ such that 
$(\yorqsp{X}{F}, \yorind{d})$ and 
$(\yorqsp{X}{F},  \yorind{e})$ are not
homeomorphic to each other. 
We give examples. 
Put $X=\rr$ and $F=\zz_{\ge 0}$.
Let $d\in \met(X)$ be the ordinary Euclidean metric and 
 $e\in \met(X)$ be a metric such that there exists 
$E\in \met(\yorop{\rr})$ with $E|_{X^{2}}=e$. 
Then the space $(\yorqsp{X}{F}, \yorind{d})$ is 
non-compact and $(\yorqsp{X}{F}, \yorind{e})$ is 
compact (see Proposition \ref{prop:qlocpt}). Thus, they are not homeomorphic to 
each other. 
On the other hand, if $F$ is compact, 
all metric quotient spaces associated with $F$ are homeomorphic to each 
other. 
\end{rmk}

\subsection{Retractions}
For a topological space $X$, 
a  subset $F$ of $X$ is said to be 
a \emph{retract} if there exists a 
continuous map $r\colon X\to F$ such that 
$r(a)=a$ for all $a\in F$. 
In this case, the map $r$ is called  a 
\emph{retraction}. 

Let $(X, d)$ be a metric space. 
For a subset $A$ of $X$, and for 
$\epsilon\in (0, \infty)$, 
we define
$\yorex(A, \epsilon)=
X\setminus \yorsetballu(A, \epsilon)$. 
Notice that $\yorex(A, \epsilon)=\{\, x\in X\mid \epsilon\le \yorrho{d}{A}(x)\, \}$ and 
$\yorex(A, \epsilon)$ is closed in $X$. 
For a subset $A$ of $X$, we denote by $\yodiam(A)$
the diameter of $A$. 
For $h\in (0, \infty)$, we say that
a subset $E$ of $X$ is \emph{$h$-separated} if 
all distinct $x, y\in E$ satisfy 
$h\le d(x, y)$. 

The next theorem is essentially 
proven by
Engelking \cite[Lemma]{MR239571}.  
Since the conditions \ref{item:r11}
and \ref{item:r22} in  the next theorem do not directly appear in 
the statement of  \cite[Lemma]{MR239571}, 
for the sake of self-containedness, 
we give a proof. 
The condition \ref{item:r22} 
plays an important role of the proof of 
Theorem \ref{thm:main1}.
\begin{thm}\label{thm:engelking}
Let $X$ be a metrizable space, 
and $F$ be a 
non-empty closed subset of $X$ with 
$\yotdim{X\setminus F}= 0$. 
Fix $d\in \met(X)$. 
Then there exists a retraction 
$r\colon X\to F$ 
satisfying the following conditions:
\begin{enumerate}[label=\textup{(D)}]
\item\label{item:r11}
For all $\epsilon\in (0, \infty)$, 
the image set  $r\left(\yorex(F, \epsilon)\right)$ is 
closed and 
discrete  in  $F$. 
\end{enumerate}
\begin{enumerate}[label=\textup{(SR)}]
\item\label{item:r22} If 
a sequence $\{x_{i}\}_{i\in \zz_{\ge 0}}$ in $X$ satisfies 
$\lim_{i\to \infty}\yorrho{d}{F}(x_{i})=0$, 
then we have $\lim_{i\to \infty}d(x_{i}, r(x_{i}))=0$. 
\end{enumerate}
\end{thm}
\begin{proof}
The idea of the construction in this proof roughly coincides 
with that of \cite[Lemma]{MR239571}, however, 
some details are different from it. 

Using Lemma \ref{lem:clopen} and induction, 
we can take a nested  sequence 
$\{V_{i}\}_{i\in \zz_{\ge 0}}$ of clopen subsets of 
$X$ satisfying  that 
\[
\yorsetballb(F, 2^{-i-1}; d)\yosub V_{i}
\]
and 
\[V_{i}\yosub \yorsetballu(F, 2^{-i}; d)\cap V_{i-1}
\]
 for all $i\in \zz_{\ge 0}$, where we put $V_{-1}=X$
(in the $i$-th step, apply Lemma \ref{lem:clopen} to 
$A=\yorsetballb(F, 2^{-i-1}; d)$ 
and 
$U=\yorsetballu(F, 2^{-i}; d)\cap V_{i-1}$). 

Note that for all $i\in \zz_{\ge 0}$, we obtain 
the next inclusions. 
\begin{enumerate}[label=\textup{(\alph*)}]
\item\label{item:reta}
$\yorex(F, 2^{-i})\yosub X\setminus V_{i}$, 
\item\label{item:retb}
$X\setminus V_{i}\yosub \yorex(F, 2^{-i-1})$.
\end{enumerate}
Put $U_{0}=X\setminus V_{0}$ and 
$U_{i}=V_{i}\setminus V_{i+1}$ for all $i\in \zz_{\ge 1}$. 
According to \cite[Corollary 1.4]{ellis1970extending} and  $\yotdim{X\setminus F}=0$, 
for each $i\in \zz_{\ge 0}$, we can find 
 a mutually disjoint  family $\{O_{s}\}_{s\in S(i)}$ of 
 clopen subsets of $X$
such that $\coprod_{s\in S(i)} O_{s}=U_{i}$ and 
$\yodiam O_{s}\le 2^{-i}$. 
We may assume  that  $S(i)\cap S(j)= \emptyset$
if $i\neq j$. 
Put $S=\coprod_{i\in \zz_{\ge 0}}S(i)$. 
Notice that 
$\coprod_{s\in S}O_{s}=X$. 
For every $i\in \zz_{\ge 0}$, 
let $P_{i}$ be a maximal  $(2^{-i})$-separated 
subset of $F$ with respect to $d$ such that 
$P_{i}\yosub P_{i+1}$ for all $i\in \zz_{\ge 0}$.
Zorn's lemma guarantees the existence of 
$\{P_{i}\}_{i\in \zz_{\ge 0}}$. 

For
every $i\in \zz_{\ge 0}$ and 
for every   $s\in S(i)$, 
put  
$R_{s}=\inf\{\, d(o, a)\mid \text{$o\in O_{s}$ and  
$a\in F$}\, \}$. 
We can  take points 
 $o_{s}\in O_{s}$  and 
 $a_{s}\in F$, 
with 
$d(o_{s}, a_{s})<R_{s}+2^{-i}$. 
Since $P_{i}$ is a maximal $(2^{-i})$-separable subset of $F$, 
we can  find a point 
$p_{s}\in P_{i}$ with  
$d(a_{s}, p_{s})<2^{-i}$. 
Then for all $i\in \zz_{\ge 0}$ and  $s\in S(i)$, we have 
\begin{align}\label{item:poR2}
d(o_{s}, p_{s})<R_{s}+2^{-i+1}. 
\end{align}
We define a map 
$r\colon X\to F$ by 
\begin{align*}
r(x)=
\begin{cases}
x & \text{$x\in F$,}\\
p_{s} & \text{if $x\in O_{s}$ for some $s\in S$.}
\end{cases}
\end{align*}
By this definition,
the map $r$ satisfies $r(a)=a$ for all $a\in F$. 

We now show the condition \ref{item:r11}. 
For every $\epsilon\in (0, \infty)$, 
take $i\in \zz_{\ge 0}$ with 
$2^{-i}\le \epsilon$. 
Then,  due to the property \ref{item:reta},   we have 
\[
r\left(\yorex(F, \epsilon)\right)
\yosub 
r\left(\yorex(F, 2^{-i})\right)
\yosub 
r(X\setminus V_{i})
=\bigcup_{k=0}^{i-1}r(U_{k})\yosub P_{i}.
\] 
Since $P_{i}$ is closed and discrete in $F$, 
so is 
$r\left(\yorex(A, \epsilon)\right)$. 
This means that 
the condition \ref{item:r11} is satisfied. 

We next show  the condition \ref{item:r22}. 
Take a sequence $\{x_{n}\}_{n\in \zz_{\ge 0}}$
with $\lim_{n\to \infty}\yorrho{d}{F}(x_{n})=0$. 
If $x_{n}\in F$, then $d(x_{n}, r(x_{n}))=0$. 
Thus, 
 we only need to consider the case where 
each $x_{n}$ belongs to $X\setminus F$. 
For each $n\in \zz_{\ge 0}$, 
take $M(n)\in \zz_{\ge 0}$ and 
$s(n)\in S(M(n))$
satisfying that  $x_{n}\in O_{s(n)}$. 
Since $\lim_{n\to \infty}\yorrho{d}{F}(x_{n})=0$, 
the inequality 
$R_{s(n)}\le \yorrho{d}{F}(x_{n})$ implies 
 $R_{s(n)}\to 0$  as 
 $n\to \infty$. 
To prove $M(n)\to \infty$ as $n\to \infty$, 
 for the sake of contradiction, suppose that 
 $M(n)\not\to \infty$ as $n\to \infty$. 
 Then 
we can take an integer $k\in \zz_{\ge 0}$ and   a strictly increasing sequence 
 $\phi\colon \zz_{\ge 0}\to \zz_{\ge 0}$ such that 
$M(\phi(n))\le k$ for all $n\in \zz_{\ge 0}$. 
Notice that
 $x_{\phi(n)}\in X\setminus V_{k+1}$ 
 for all 
 $n\in \zz_{\ge 0}$. 
 The inclusion 
 $X\setminus V_{k+1}\yosub \yorex(F, 2^{-k-2})$ (see \ref{item:retb})
  implies that 
 $2^{-k-2}\le \yorrho{d}{F}(x_{\phi(n)})$ for all $n\in \zz_{\ge 0}$. 
 This contradicts the equality that $\lim_{i\to \infty}\yorrho{d}{F}(x_{i})=0$. 
Thus $M(n)\to \infty$
 as $n\to \infty$. 
From $\yodiam O_{s(n)}\le 2^{-M(n)}$, 
$x_{n}\in O_{s(n)}$,  and \eqref{item:poR2}, it follows that  
\begin{align*}
&d(x_{n}, r(x_{n}))
=d(x_{n}, p_{s(n)})
\le d(x_{n}, o_{s(n)})+d(o_{s(n)}, p_{s(n)})\\
&< \yodiam(O_{s(n)}) + R_{s(n)}+2^{-M(n)+1}
\le R_{s(n)}+2^{-M(n)+2}\to 0
\end{align*}
 as $n\to \infty$. 
This proves the condition \ref{item:r22}. 

It remains to show that $r$ is continuous. 
Take sequence $\{x_{i}\}_{i\in \zz_{\ge 0}}$ in $X$ 
and $a\in X$ such that 
$x_{i}\to a$. 
If $a\in X\setminus F$, then the point 
$a$ belongs to some 
$O_{s}$. Thus $x_{i}\in O_{s}$ for all sufficiently large 
 $i\in \zz_{\ge 0}$. 
Hence 
$r(x_{i})\to p_{s}(=r(a))$. 
If $a\in F$, then 
$\lim_{i\to \infty} \yorrho{d}{F}(x_{i})=0$. 
The condition \ref{item:r22} implies that  
$\lim_{i\to \infty}d(x_{i}, r(x_{i}))=0$. 
Namely, 
 $r(x_{i})\to a(=r(a))$. 
This finishes the proof. 
\end{proof}

\begin{rmk}
The name of the condition \ref{item:r22} in Theorem \ref{thm:engelking} stands for 
``Self-Reference.''
\end{rmk}

To show that 
Theorem \ref{thm:engelking} certainly 
covers \cite[Lemma]{MR239571}, we prove the following lemma, 
which is not  
used to prove our main results
in this paper. 

\begin{lem}\label{lem:closedretract}
Let $X$ be a metrizable space, 
and $F$ be a 
non-empty closed subset of $X$ with 
$\yotdim{X\setminus F}= 0$. 
If a retraction 
$r\colon X\to F$ 
and $d\in \met(X)$ satisfy the conditions 
\ref{item:r11} and \ref{item:r22} in 
Theorem \ref{thm:engelking}, 
then $r$ is a closed map. 
\end{lem}
\begin{proof}
Take an arbitrary closed subset of 
$A$ of  
$X$ and 
take an arbitrary point 
$y\in \cl(r(A))$. 
We shall show that 
$y\in r(A)$. 
If 
$y$ is a isolated point of $\cl(A)$, 
then 
$y\in r(A)$. 
We only need to consider the case where  
$y$ is not isolated in $\cl(A)$. 
In this case, 
there exists a sequence 
$\{x_{i}\}_{i\in \zz_{\ge 0}}$ 
in 
$A$
such that $r(x_{i})\to y$ 
as $i\to \infty$
and 
$r(x_{i})\neq r(x_{j})$ 
for all distinct 
$i, j\in \zz_{\ge 0}$. 

For the sake of contradiction, 
suppose that 
there exists 
$\epsilon \in (0, \infty)$ 
such that 
$x_{i}\in \yorex(F, \epsilon)$ for all 
$i\in \zz_{\ge 0}$. 
From the condition \ref{item:r11}, 
it follow that 
the set 
$\{\, r(x_{i})\mid i\in \zz_{\ge 0}\, \}$
is closed and  discrete in  $F$. 
This is a contradiction to the assumptions
that 
$r(x_{i})\to y$ 
as 
$i\to \infty$ 
and  
$r(x_{i})\neq r(x_{j})$ 
for all distinct 
$i, j\in \zz_{\ge 0}$. 
Thus we can take subsequence 
$\{z_{i}\}_{i\in \zz_{\ge 0}}$ 
of 
$\{x_{i}\}_{i\in \zz_{\ge 0}}$ 
such that 
$\lim_{i\to 0}\yorrho{d}{F}(z_{i})=0$. 
The condition \ref{item:r22} 
guarantees 
 $\lim_{i\to \infty}d(z_{i}, r(z_{i}))=0$. 
Hence $z_{i}\to y$ as $i\to \infty$. 
This yields $y\in A\cap F$,  and hence
$r(y)=y$. Thus $y\in r(A)$. 
Therefore we conclude that 
$r$ is a closed map. 
\end{proof}

\begin{rmk}\label{rmk:ret}
In \cite[Theorem 2.9]{brodskiy2007dimension}, 
for every  ultrametric space $(X, d)$, and for 
every 
closed subset $F$, 
Brodskiy--Dydak--Higes--Mitra constructed a 
Lipschitz retraction from $X$ to $F$. 
To remark that their retraction satisfies the 
conditions \ref{item:r11} and \ref{item:r22}, 
we briefly review the construction. 
Let $(X, d)$ be an ultrametric space, 
and $A$ be a closed subset of $X$. 
Fix $\tau\in (1, \infty)$. 
Take an order $\preceq$ on $X$ such that 
it is well-ordered on every bounded subset of $X$. 
Such a order can be obtained by 
gluing  well-orderings on $\{\, x\in X\mid k\le d(x, y)<k+1\, \}$ together, where $k\in \zz_{\ge 0}$. 
For $x\in X$, we put 
$\stst{A}{x}=\{\, a\in A\mid d(x, a)\le\tau \cdot \yorrho{d}{A}(x) \, \}$. 
We define a map $r\colon X\to A$ by 
$r(x)=\min_{\preceq}\stst{A}{x}$; namely 
$r(x)$ is the least element of $\stst{A}{x}$ with respect to 
$\preceq$. 
Then $r$ is a $(\tau^{2})$-Lipschitz retraction
(see \cite[Theorem 2.9]{brodskiy2007dimension}). 
By the construction
(especially, $\stst{A}{x}$),  the map $r$ satisfies the condition 
\ref{item:r22}. 
We now prove the condition \ref{item:r11}. 
We shall show that 
for all $\epsilon\in (0, \infty)$, 
if $x, y\in \yorex(A, \epsilon)$ satisfies 
$d(r(x), r(y))<\epsilon$, 
then $r(x)=r(y)$. 
In this setting, since $d$ satisfies the strong triangle inequality and $d(r(x), r(y))<d(x, r(x))$, 
we have $d(x, r(x))=d(x, r(y))$,  and 
hence $r(y)\in \stst{A}{x}$.  
By the definition of $r$, we obtain  $r(x)\preceq r(y)$. 
Similarly, we also obtain 
 $r(y)\preceq r(x)$. 
Hence $r(x)=r(y)$. 
Therefore $r(\yorex(A, \epsilon))$ is 
$\epsilon$-separated in $(X, d)$, and hence 
the condition \ref{item:r11} is satisfied. 
Note that 
due to Lemma  \ref{lem:closedretract}, 
the retraction $r$ is a closed map. 
Based on the fact that $r$ satisfies the 
condition \ref{item:r22}, 
as far as we consider ultrametrizable  spaces, 
in the arguments discussed below, 
we can use Brodskiy--Dydak--Higes--Mitra's retraction instead of 
Engelking's one. 
\end{rmk}

\section{Embedding into a product space}\label{sec:proof}
This section is devoted to proving 
our first main result: 
\begin{thm}\label{thm:main1}
Let $X$ be a metrizable  space, 
 and 
$F$ be a non-empty closed subset of $X$ with 
$\yotdim{X\setminus F}=0$. 
Fix $d\in \met(X)$, and 
we consider that $\yorqsp{X}{F}$ is equipped with 
the topology generated by $\yorind{d}$. 
Let $r\colon X\to F$ be a retraction 
stated in Theorem \ref{thm:engelking}
associated with $d$. 
Define a map $\yormainmap\colon X\to F\times \yorqsp{X}{F}$ by 
$\yormainmap(x)=(r(x), \yorqmap(x))$. 
Then the map 
$\yormainmap$
satisfies  the following conditions:
\begin{enumerate}[label=\textup{(P\arabic*)}]
\item\label{item:main11}
We have
$\yormainmap(a)=(a, \yorppp{F})$ for all $a\in \yorclosed$. 
\item\label{item:main12}
The map $\yormainmap$ is a topological embedding and 
$\yormainmap(X)$ is a closed subset of 
$\yorclosed\times \yorqsp{X}{F}$. 
\end{enumerate}
\end{thm}
\begin{proof}
First note that the map $\yormainmap$ is continuous and 
injective. 
The condition \ref{item:main11} is deduced 
from the definition of $\yormainmap$. 
To prove the condition \ref{item:main12}, 
it suffices to show that 
$\yormainmap$ is a closed map. 
Take an arbitrary closed subset $A$ of $X$. 
To verify  that $\yormainmap(A)$ is 
closed in $F\times \yorqsp{X}{F}$, 
take $y\in \cl(\yormainmap(A))$. 
If 
$y$ is isolated in $ \cl(\yormainmap(A))$, 
then $y\in \yormainmap(X)$. 
We may assume that 
$y$ is non-isolated in 
$ \cl(\yormainmap(A))$. 
Then we can find a sequence 
$\{x_{i}\}_{\in \zz_{\ge 0}}$ in $A$
such that 
$\yormainmap(x_{i})\to y$
as
$i\to \infty$. 
Put $y=(z, s)$, where $z\in F$ and $s\in \yorqsp{X}{F}$. 
Remark that  in this situation, 
we have $r(x_{i})\to z$ and 
$\yorqmap(x_{i})\to s$ as 
$i\to \infty$. 
We divide the proof into two cases. 

Case 1.~($s=\yorqmap(w)$ for some $w\in X\setminus F$): 
Since 
$\yorqmap|_{X\setminus F}\colon 
X\setminus F\to \yorqsp{X}{F}\setminus \{\yorppp{F}\}$ is a homeomorphism (see 
\ref{item:theta3} in 
Proposition 
\ref{prop:inducedmetric}), 
the convergence 
$\yorqmap(x_{i})\to \yorqmap(w)(=s)$ implies 
that $x_{i}\to w$ as 
$i\to \infty$. Thus $w\in A$. 
Using the continuity of $\yormainmap$,  
we obtain 
$y=\yormainmap(w)$,  and hence 
$y\in \yormainmap(A)$. 

Case 2.~($s=\yorppp{F}$): 
In this case, 
we have $\yorind{d}(x_{i}, \yorppp{F})\to 0$, 
and hence $\yorrho{d}{F}(x_{i})\to 0$ as 
$i\to \infty$. 
The condition \ref{item:r22} in Theorem 
\ref{thm:engelking} implies that 
$d(x_{i}, r(x_{i}))\to 0$ as $i\to \infty$. 
Since 
$r(x_{i})\to z$ as $i\to \infty$, we have 
$x_{i}\to z$ as $i\to \infty$. 
Thus $z\in A$. 
Using the continuity of 
$\yormainmap$, we have 
$y= \yormainmap(z)$,  and hence 
$y\in \yormainmap(A)$. 
Therefore 
$\yormainmap(A)$ is closed in $F\times \yorqsp{X}{F}$. 
This finishes the proof. 
\end{proof}

\begin{rmk}
By Proposition \ref{prop:inducedzero} 
the space $\yorqsp{X}{F}$ 
in Theorem 
\ref{thm:main1}
satisfies 
$\yotdim{\yorqsp{X}{F}}=0$. 
\end{rmk}
\begin{rmk}
Let $X$ be a topological space,  and $F$ be a 
closed subset of $X$. 
There are 
some functional   decompositions
similar to Theorem \ref{thm:main1}. 
Namely, under certain  assumptions, 
in some sense, 
a function space of $X$ can be factorized 
into a product of a function space of $F$ and 
a function space of a quotient space $X/F$
(see for example \cite[Corollary 2.3]{MR972840} and 
\cite[Lemma 2.2]{MR3356002}). 
\end{rmk}

\section{Applications}\label{sec:app}
In this section, we prove applications of 
Theorem \ref{thm:main1}. 
\subsection{Fractal dimensions}
We first review the fractal dimensions. 
Most of definitions and 
statements on the fractal dimensions  in this subsection 
are shared 
with the author's preprint 
\cite{ishiki2021dimension}.

In this paper, we mainly deal with 
the Hausdorff dimension
$\yorhdim{X, d}$, 
the packing dimension
$\yorpdim{X, d}$, 
the upper box dimension
$\yorubdim{X, d}$, 
and
the Assouad dimension
$\yoradim{X, d}$.

\subsubsection{The Hausdorff dimension}
Let $(X, d)$ be a metric space. 
For $\delta \in (0,\infty)$, 
we denote by $\mathbf{F}_{\delta}(X, d)$ the set of all subsets of $X$ with diameter smaller than $\delta$. 
For  $s\in [0,\infty)$, and $\delta\in (0, \infty)$, 
we define the measure $\mathcal{H}_{\delta}^{s}$ 
on $X$ as
\[
\mathcal{H}_{\delta}^s(A)=
\inf\left\{\, 
\sum_{i=1}^{\infty}\yodiam(A_i)^s\ 
\middle| \ A\yosub\bigcup_{i=1}^{\infty}A_i,\ A_i\in \mathbf{F}_{\delta}(X, d)\,\right\}.
\]
For $s\in (0, \infty)$
we define 
the \emph{$s$-dimensional Hausdorff measure 
 $\mathcal{H}^s$}
on $(X, d)$  as 
$\mathcal{H}^s(A)=
\sup_{\delta\in (0,\infty)}\mathcal{H}_{\delta}^s(A)$. 
We denote by $\hdim(X, d)$ 
the \emph{Hausdorff dimension of $(X, d)$} defined as
$\hdim(X, d)=\sup \{\,s\in[0,\infty)\mid 
\mathcal{H}^s(X)=\infty\,\}
             =\inf\{\,s\in [0,\infty)\mid 
             \mathcal{H}^s(X)=0\,\}$.

\subsubsection{The packing dimension}
Let $(X, d)$ be a metric space. 
For a subset $A$ of $X$, and for $\delta\in (0, \infty)$, 
we denote by $\mathbf{Pa}_{\delta}(A)$ the set of all 
finite or countable sequence $\{r_{i}\}_{i=1}^{N} (N\in \zz_{\ge 1}\cup\{\infty\})$ in $(0, \delta)$
 for which 
there exists a sequence $\{x_{i}\}_{i=1}^{N}$ in $A$ such 
that if $i\neq j$, then we have 
$B(x_{i}, r_{i})\cap B(x_{j}, r_{j})=\emptyset$. 
For $s\in [0, \infty)$,  $\delta\in (0, \infty)$,
and  a subset $A$ of $X$,  
we define the quantity  $\widetilde{\mathcal{P}}_{\delta}^{s}(A)$ by 
\[
\widetilde{\mathcal{P}}_{\delta}^{s}(A)=
\sup\left\{\, \sum_{i=1}^{N}r_{i}^{s}\ 
\middle| \ \{r_{i}\}_{i=1}^{N} \in \mathbf{Pa}_{\delta}(A)
\, \right\}. 
\]
We then define the \emph{$s$-dimensional pre-packing measure 
$\widetilde{\mathcal{P}}_{0}^{s}$} on $(X, d)$ as 
$\widetilde{\mathcal{P}}_{0}^{s}(A)
=\inf_{\delta\in (0, \infty)}
\widetilde{\mathcal{P}}_{\delta}^{s}$, 
and 
we define the 
\emph{$s$-dimensional packing measure } 
$\mathcal{P}^{s}$ on 
$(X, d)$
 as
\[
\mathcal{P}^{s}(A)=
\inf\left\{\, \sum_{i=1}^{\infty}\widetilde{\mathcal{P}}_{0}^{s}(S_{i})\ 
\middle|\  A\yosub \bigcup_{i=1}^{\infty}S_{i}
\, \right\}. 
\]
We denote by $\pdim(X, d)$ 
the \emph{packing  dimension of $(X, d)$},  which is  defined as
$\pdim(X, d)=\sup \{\,s\in[0,\infty)\mid 
\mathcal{P}^s(X)=\infty\,\}
             =\inf\{\,s\in [0,\infty)\mid 
             \mathcal{P}^s(X)=0\,\}$.

\subsubsection{Box dimensions}
For a metric space $(X, d)$, and for $r\in (0, \infty)$, 
 a subset $A$ of $X$ is said to be an 
 \emph{$r$-net} if 
 $X=\bigcup_{a\in A}B(a, r)$. 
 We denote by $\mnum_{d}(X, r)$ the least  cardinality  of  $r$-nets of $X$.  
We define 
the \emph{upper box dimension $\ubdim(X, d)$} 
by 
\begin{align*}
\ubdim(X, d)&=\limsup_{r\to 0}\frac{\log \mnum_{d}(X, r)}{-\log r}. 
\end{align*}

\subsubsection{The Assouad dimension}

For a metric space $(X, d)$, 
we denote by  $\adim(X, d)$ 
the \emph{Assouad dimension 
of $(X, d)$}
defined  by the infimum of all 
$\lambda\in (0, \infty)$ for which 
there exists $C\in (0, \infty)$ such that 
for all $R, r\in (0, \infty)$ with $r<R$ and for all $x\in X$
we have 
$\mnum_{d}(B(x, R), r)\le 
C\left(R/r\right)^{\lambda}$. 
If such $\lambda$ does not exist, we define 
$\adim(X, d)=\infty$. 
We say that a metric space is \emph{doubling} if 
its Assouad dimension is finite. 

The next theorem is deduced  from 
\cite{Szpilrajn1937}, \cite[Lemma 2.4.3]{fraser2020assouad},  and 
\cite[Chapters 2 and  3]{falconer2004fractal}. 
\begin{thm}\label{thm:diminequalities}
The following statements hold true:
\begin{enumerate}
\item 
If $(X, d)$ is  a  bounded metric space, 
then we have 
\[
\hdim(X, d)\le \pdim(X, d)\le \ubdim(X, d)\le \adim(X, d). 
\]
\item 
If $(X, d)$ is a metric space, then we have 
\[
\hdim(X, d)\le \pdim(X, d)\le \adim(X, d). 
\]
\end{enumerate}
\end{thm}

\begin{rmk}
If $(X, d)$ is unbounded, then 
the inequality $\ubdim(X, d)\le \adim(X, d)$ is 
not always true 
since in this setting,  $\ubdim(X, d)=\infty$. 
For example, the Euclidean space $(\rr, d_{\rr})$
satisfies that $\ubdim(\rr, d_{\rr})=\infty$ and 
$\adim(\rr, d_{\rr})=1$.
\end{rmk}

From the definitions of 
fractal dimensions, 
we deduce the next proposition. 
\begin{prop}\label{prop:subsub}
Let $(X, d)$ be a metric space, and 
$A$ be a subset of $X$. 
Let $\mathbf{Dim}$ be any one of 
$\hdim$,  $\pdim$, 
$\ubdim$, or $\adim$. 
Then we have the inequality 
$\mathbf{Dim}(A, d|_{A^{2}})\le 
\mathbf{Dim}(X, d)$. 
\end{prop}

We denote by $\yorlset$
the set of all $(a_{1}, a_{2}, a_{3}, a_{4})\in [0, \infty]^{4}$
such that $a_{1}\le a_{2}\le a_{3}\le a_{4}$. 
We also denote by 
$\yorrset$ the set of all $(b_{1}, b_{2}, b_{3})\in [0, \infty]^{3}$ such that 
$b_{1}\le b_{2}\le b_{3}$.

Based on Theorem \ref{thm:diminequalities}, 
we define sets of metrics with prescribed fractal dimensions. 
Let $X$ be a metrizable space. 
For $\yorbd{a}=(a_{1}, a_{2}, a_{3}, a_{4})\in \yorlset$, 
let $\yordimset{X}{\yorbd{a}}$ be the 
set of all bounded metric $d\in \met(X)$ 
satisfying that
$\yorhdim{X, d}=a_{1}$, 
$\yorpdim{X, d}=a_{2}$, 
$\yorubdim{X, d}=a_{3}$, 
and 
$\yoradim{X, d}=a_{4}$. 
For $\yorbd{b}=(b_{1}, b_{2}, b_{3})\in \yorrset$, 
let $\yordimsetq{X}{\yorbd{b}}$ be the 
set of all $d\in \met(X)$ with 
$\yorhdim{X, d}=b_{1}$, 
$\yorpdim{X, d}=b_{2}$, 
and 
$\yoradim{X, d}=b_{3}$. 
Note that $\yordimsetq{X}{\yorbd{b}}$ can 
contain unbounded metrics. 
For the sake of convenience, 
we define 
$\yorzeron=(0, 0, 0, 0)\in \yorlset$
and 
$\yorzeroz=(0, 0, 0)\in \yorrset$.

\begin{lem}\label{lem:nullzero}
If a separable metrizable space  $X$ 
  satisfies  $\yotdim{X}=0$, 
then there exists a metric $v\in \yordimset{X}{\yorzeron}$. 
In this case, we also have 
$v\in \yordimsetq{X}{\yorzeroz}$. 
\end{lem}
\begin{proof}
Due to Proposition \ref{prop:sepembzero}, 
we can consider that $X$ is a subset of a Cantor 
space $\Gamma$. 
By \cite[theorem 1.8]{Ishiki2019} or \cite[Lemma 3.7]{ishiki2021dimension}, there exists a metric 
$d\in \met(\Gamma)$
with $\yoradim{\Gamma, d}=0$. 
Theorem \ref{thm:diminequalities} implies that 
$d\in \yordimset{X}{\yorzeron}$. 
Put $v=d|_{X^{2}}$. 
Then, according to  
Proposition \ref{prop:subsub}, 
we have 
$v\in \yordimset{X}{\yorzeron}$. 
\end{proof}

The proof of the following can be found  in \cite{falconer2004fractal} and 
\cite{fraser2020assouad}. 
\begin{prop}\label{prop:productdim}
Let $(X, d)$ and $(Y, e)$ be metric spaces, 
and $k=1$ or $k=\infty$. 
Then the following statements hold true: 
\begin{enumerate}
\item 
Let $\mathbf{Dim}$ be any one of $\pdim$, 
$\ubdim$, or $\adim$. Then we have 
\[
\mathbf{Dim}(X\times Y, d\times_{k}e)\le 
\mathbf{Dim}(X, d)+\mathbf{Dim}(Y, e).
\] 
\item 
$\hdim(X\times Y, d\times_{k}e)\le \hdim(X, d)+
\pdim(Y, e)$. \label{item:haus2}
\end{enumerate}
\end{prop}

\begin{rmk}
In (\ref{item:haus2}) of Proposition \ref{prop:productdim}, 
the packing dimension 
$\pdim(Y, e)$ in the right hand side 
can not be replaced with the Hausdorff dimension
in general. 
Considering only the Hausdorff dimension, 
we obtain the opposite inequality 
$\hdim(X, d)+
\hdim(Y, e)\le \hdim(X\times Y, d\times_{k}e)$. 
\end{rmk}

Recall that 
a subset $S$ of $[0, \infty)$ is 
said to be 
\emph{characteristic} if 
$0\in S$ and for all $t\in (0, \infty)$, 
there exists $s\in S\setminus \{0\}$ such that 
$s\le t$.

For an ultrametrizable space $X$, 
for a characteristic subset $S$ of $[0, \infty)$, 
for $\yorbd{a}\in \yorlset$ and 
$\yorbd{b}\in \yorrset$, 
we put 
$\yordimsetu{X}{S}{\yorbd{a}}
=\ult{X}{S}\cap \yordimset{X}{\yorbd{a}}$ and 
$\yordimsetqu{X}{S}{\yorbd{b}}
=\ult{X}{S}\cap \yordimsetq{X}{\yorbd{b}}$. 
Proposition \ref{prop:productdim} implies the 
following corollary. 
\begin{cor}\label{cor:prodprod}

Take arbitrary points $\yorbd{a}\in \yorlset$ and 
$\yorbd{b}\in \yorrset$. 
Let $X$ and $Y$ be metrizable spaces, 
 $S$ be a characteristic subset of $[0, \infty)$. 
Then the following statements hold true: 
\begin{enumerate}
\item\label{item:corp:1}
If $d\in \yordimset{X}{\yorbd{a}}$ and 
$v\in \yordimset{Y}{\yorzeron}$, 
then $d\times_{1} v\in \yordimset{X\times Y}{\yorbd{a}}$. 
\item\label{item:corp:2}
If $d\in \yordimsetq{X}{\yorbd{b}}$ and 
$v\in \yordimsetq{Y}{\yorzeroz}$, 
then $d\times_{1} v\in \yordimsetq{X\times Y}{\yorbd{b}}$. 
\item\label{item:corp:3}
If $d\in \yordimsetu{X}{S}{\yorbd{a}}$ and 
$v\in \yordimsetu{Y}{S}{\yorzeron}$, 
then we have $d\times_{\infty} v\in \yordimsetu{X\times Y}{S}{\yorbd{a}}$. 
\item\label{item:corp:4}
If $d\in \yordimsetqu{X}{S}{\yorbd{b}}$ and 
$v\in \yordimsetqu{Y}{S}{\yorzeroz}$, 
then we have  $d\times_{\infty} v\in \yordimsetqu{X\times Y}{S}{\yorbd{b}}$. 
\end{enumerate}
\end{cor}

\subsection{Extensors}

For a metric space $(X, d)$, 
we denote by $\yorcompset(X)$
the set of all complete metrics in $\met(X)$ . 
We also 
denote by $\yorpropset(X)$
the set of all proper metrics in $\met(X)$ .
For $\eta\in (0, \infty)$, 
a subset $A$ of a metric space $(X, d)$ is
\emph{$\eta$-dense} if for all $x\in X$, 
there exists $a\in A$ with $d(x, a)\le \eta$. 
Let $X$, $Y$, and $Z$ be sets. 
Let 
$f\colon X\to Y$  and 
$d\colon Y\times Y\to Z$ be maps. 
Then, we define a map
$f^{*}d\colon X\times X\to Z$ by 
$f^{*}d(x, y)=d(f(x), f(y))$. 
For a metrizable space $X$, and 
for metrics  $d, e\in \met(X)$, 
we also define $d\lor e\in \met(X)$
by $(d\lor e)(x, y)=d(x, y)\lor e(x, y)$. 

The next is our second main result:
\begin{thm}\label{thm:subARC}
Let $X$ be a metrizable space, 
and $F$ be a non-empty 
closed subset of $X$ with 
$\yotdim{X\setminus F}=0$. 
Let $\yorfmet\in \met(X)$ and 
we consider that $\yorqsp{X}{F}$ is 
equipped with the topology generated by 
$\yorind{\yorfmet}$. 
Let $\yorflatmap\colon X\to F\times \yorqsp{X}{F}$
be a topological embedding 
stated  in Theorem \ref{thm:main1}
associated with $\yorfmet$. 
Let $\yorvmet\in \met(\yorqsp{X}{F})$. 
Define a map  
$\yorsubmap\colon \met(F)\to \met(X)$
by $\yorsubmap(d)=\yorflatmap^{*}(d\times_{1}\yorvmet)$. 
Then the map $\yorsubmap$
 satisfies that 
\begin{enumerate}[label=\textup{(\arabic*)}]
\item\label{item:sub1:11}
for all $d\in \met(F)$, we have 
$\yorsubmap(d)|_{F^{2}}=d$; 
\item\label{item:sub1:22}
for all $d, e\in \met(F)$, 
we have 
$\metdis_{X}(\yorsubmap(d), \yorsubmap(e))
=\metdis_{F}(d, e)$ (i.e.,  $\yorsubmap$ is 
an isometric map); 
\item\label{item:sub1:33}
if $d, e\in \met(F)$ satisfies $d(a, b)\le e(a, b)$ for all 
$a, b\in F$,  then $\yorsubmap(d)(x, y)\le \yorsubmap(e)(x, y)$ for all $x, y\in X$;
\item\label{item:sub1:44}
for all $d, e\in \met(F)$
we have $\yorsubmap(d\lor e)=
\yorsubmap(d)\lor \yorsubmap(e)$. 
\end{enumerate}
In addition, 
the following statements are true: 
\begin{enumerate}[label=\textup{(A\arabic*)}]
\item\label{item:subARC:cm}
 If $X$ is completely metrizable, 
 then metrics 
 $\yorfmet\in \met(X)$ and 
$\yorvmet\in \met(\yorqsp{X}{F})$
can be chosen  so that 
$\yorsubmap(\yorcompset(F))\yosub 
\yorcompset(X)$. 
\item\label{item:subARC:loccpt}
If $X$ is $\sigma$-compact locally compact, 
then we can choose 
$\yorfmet\in \met(X)$ and 
$\yorvmet\in \met(\yorqsp{X}{F})$ 
so that 
$\yorqsp{X}{F}$ is compact and 
we have
$\yorsubmap(\yorpropset(F))\yosub\yorpropset(X)$. 
\item\label{item:subARC:sep}
If $X$ is separable, then
$\yorvmet\in \met(\yorqsp{X}{F})$ can be 
chosen  to satisfy that
for all $\yorbd{a}\in \yorlset$ and 
$\yorbd{b}\in \yorrset$, we have 
$\yorsubmap(\yordimset{F}{\yorbd{a}})
\yosub \yordimset{X}{\yorbd{a}}$
 and 
 $\yorsubmap(\yordimsetq{F}{\yorbd{b}})
\yosub \yordimsetq{X}{\yorbd{b}}$
\item\label{item:subARC:dense}
If we fix  $\eta \in (0, \infty)$, 
then we can choose $\yorvmet\in \met(\yorqsp{X}{F})$
so that  for all $d\in \met(F)$, 
the set $F$ is $\eta$-dense in 
$(X, \yorsubmap(d))$.
 
\end{enumerate}
\end{thm}

\begin{proof}
Put 
$\yorflatmap(x)=(\yorflatmapx(x), \yorflatmapy(x))$, where $r$ is a retraction stated in Theorem \ref{thm:engelking} associated with 
$\yorfmet\in \met(X)$, and $\yorqmap\colon X\to \yorqsp{X}{F}$ is the map defined in Section \ref{sec:pre}. 
In this setting, for all $d\in \met(X)$, 
we obtain 
$\yorsubmap(d)(x, y)=
d(r(x), r(y))+v(\yorqmap(x), \yorqmap(y))$. 
First we confirm that 
for all $d\in \met(F)$, we have 
$\yorsubmap(d)\in \met(X)$, 
which follows from the fact that 
$\yormainmap$ is a topological embedding and 
$d\times_{1}v\in \met(F\times \yorqsp{X}{F})$. 

The condition \ref{item:sub1:11}
is deduced 
from the condition  \ref{item:main11} in 
Theorem \ref{thm:main1}. 

To prove the condition \ref{item:sub1:22}, take
$d, e\in \met(F)$. 
Then by the definition of the $\ell^{1}$-product ``$\times_{1}$'', 
for all $x, y\in X$, 
 we have 
\begin{align*}
|\yorsubmap(d)(x, y)-\yorsubmap(e)(x, y)|
=|d(\yorflatmapx(x), \yorflatmapx(y))
-e(\yorflatmapx(x), \yorflatmapx(y))|
\end{align*}
Thus the property \ref{item:main11} implies the 
condition \ref{item:sub1:22}. 

The condition \ref{item:sub1:33} follows from 
the definition of the $\ell^{1}$-product ``$\times_{1}$''.
 
From the fact that
$(a+c)\lor (b+c)=(a\lor b)+c$ for all 
$a, b, c\in [0, \infty)$, 
we obtain the condition 
\ref{item:sub1:44}. 
Therefore  all the properties 
\ref{item:sub1:11}--\ref{item:sub1:44}
are satisfied. 
It remains to show the additional  statements
\ref{item:subARC:cm}--\ref{item:subARC:dense}. 

To show the statement \ref{item:subARC:cm}, 
assume that  $X$ is complete metrizable. 
Then we can choose a complete metric 
$\yorfmet\in \met(X)$.
In this situation, 
the space
 $\yorqsp{X}{F}$ is  complete metrizable
(see Proposition \ref{prop:inducedcomplete}). 
 Thus we can take 
$v\in \met(\yorqsp{X}{F})$ as a complete metric. 
If $d\in \yorcompset(F)$, then 
the condition \ref{item:main12} in 
Theorem \ref{thm:main1}
 implies that 
$\yorsubmap(d)=\yorflatmap^{*}(d\times_{1}v)$ is 
complete. 
Therefore the condition \ref{item:subARC:cm} is satisfied.

Next we prove the statement \ref{item:subARC:loccpt}. 
Assume that $X$ is $\sigma$-compact and locally compact. 
Choose $\yorfmet\in \met(X)$ 
such that there exists a $D\in \met(\yorop{X})$ with 
$D|_{X}=\yorfmet$. 
Then  Proposition \ref{prop:qlocpt} implies that 
$\yorqsp{X}{F}$ is  compact. 
Fix an arbitrary metric $\yorvmet\in \met(\yorqsp{X}{F})$. 
Take $d\in \yorpropset(X)$. 
From Proposition \ref{lem:prodproper}, 
the condition \ref{item:main12}, 
and
the fact that 
a closed metric subspace of a proper metric space is proper, 
it follows that $\yorsubmap(d)=\yorflatmap^{*}(d\times_{1}v)$ is a proper metric. 

To verify  the statement \ref{item:subARC:sep}, 
assume that $X$ is separable. 
In this case, the space  $\yorqsp{X}{F}$  is 
separable and $\yotdim{\yorqsp{X}{F}}=0$
(see the statement \ref{item:26:sep} in Proposition \ref{prop:sepcpt},  and Proposition 
\ref{prop:inducedzero}). 
Lemma \ref{lem:nullzero} enables us to 
choose $\yorvmet\in \met(X)$ so that 
$\yorvmet\in \yordimset{\yorqsp{X}{F}}{\yorzeron}$. 
In this setting, we also have 
$\yorvmet\in \yordimsetq{\yorqsp{X}{F}}{\yorzeroz}$. 
According to 
Corollary \ref{cor:prodprod} and 
Proposition \ref{prop:subsub}, 
we conclude that 
$\yorsubmap(\yordimset{F}{\yorbd{a}})
\yosub \yordimset{X}{\yorbd{a}}$
 and 
 $\yorsubmap(\yordimsetq{F}{\yorbd{b}})
\yosub \yordimsetq{X}{\yorbd{b}}$.

We next show the statement \ref{item:subARC:dense}. 
Fix $\eta\in (0, \infty)$. 
Then we can take $v\in \met(\yorqsp{X}{F})$ so that 
$v(x, y)\le \eta$ for all $x, y\in \yorqsp{X}{F}$ 
(if necessary, replace $v$ with $\min\{v, \eta\}$). 
Then,  for all $x\in X$, 
due to  the definition of $\yorsubmap$ and 
$r(r(x))=r(x)$,  we obtain 
$\yorsubmap(d)(x, r(x))
=d(\yorflatmapx(x), \yorflatmapx(x))+
v(\yorflatmapy(x), \yorflatmapy(r(x)))
=v(\yorflatmapy(x), \yorflatmapy(r(x)))\le \eta$. 
Since $r(x)\in F$,  
we conclude that  the set $F$ is $\eta$-dense in 
$(X, \yorsubmap(d))$. 
This complete the proof of 
Theorem \ref{thm:subARC}. 
\end{proof}

For an ultrametrizable space $X$ and 
for a characteristic subset $S$ of $[0, \infty)$, 
we  put 
$\yorcompsetu(X; S)=\yorcompset(X)\cap 
\ult{X}{S}$ and 
$\yorpropsetu(X; S)=
\yorpropset(X)\cap \ult{X}{S}$. 

As an ultrametric version of Theorem \ref{thm:subARC}, 
we obtain 
 our third main result:
\begin{thm}\label{thm:subNARC}
Let $X$ be an ultrametrizable space, 
and $F$ be a non-empty 
closed subset of $X$.  
Let $\yorfmet\in \met(X)$ and 
we consider that $\yorqsp{X}{F}$ is 
equipped with the topology generated by 
$\yorind{\yorfmet}$. 
Let $\yorflatmap\colon X\to F\times \yorqsp{X}{F}$
be a topological embedding
stated 
 in Theorem \ref{thm:main1}. 
 Fix   a characteristic subset
 $S$ of 
$[0, \infty)$,  and 
let $\yorvmet\in \ult{\yorqsp{X}{F}}{S}$. 
Define a map  
$\yorsubmapq\colon \met(F)\to \met(X)$
by $\yorsubmapq(d)=\yorflatmap^{*}(d\times_{\infty}\yorvmet)$. 
Then the map
$\yorsubmapq$ satisfies
that 
\begin{enumerate}[label=\textup{(\arabic*)}]
\item\label{item:sub2:11}
for all $d\in \met(F)$, we have 
$\yorsubmapq(d)|_{F^{2}}=d$; 
\item\label{item:sub2:22}
for all $d, e\in \met(F)$, we have
$\umetdis_{X}^{S}(\yorsubmapq(d), \yorsubmapq(e))
=\umetdis_{F}^{S}(d, e)$
(i.e., $\yorsubmapq$ is an isometric map); 
\item\label{item:sub2:33}
if $d, e\in \met(F)$ satisfy $d(a, b)\le e(a, b)$ for all 
$a, b\in F$,  then 
$\yorsubmapq(d)(x, y)\le \yorsubmapq(e)(x, y)$
for all $x, y\in X$;
\item\label{item:sub2:44}
for all $d, e\in\met(F)$
we have $\yorsubmapq(d\lor e)=
\yorsubmapq(d)\lor \yorsubmapq(e)$;
\item\label{item:sub2:55}
we have $\yorsubmapq(\ult{F}{S})\yosub 
\ult{X}{S}$. 
\end{enumerate}

In addition, 
the following statements are true: 
\begin{enumerate}[label=\textup{(N\arabic*)}]
\item\label{item::cm}
 If $X$ is completely metrizable, 
 then
we can choose $\yorfmet\in \met(X)$ and 
$\yorvmet\in \met(\yorqsp{X}{F})$ 
so that 
$\yorsubmapq(\yorcompset(F))\yosub
\yorcompset(X)$. In this case, we also have 
$\yorsubmapq(\yorcompsetu(F; S)\yosub
\yorcompsetu(X; S)$. 
\item\label{item::loccpt}
If $X$ is $\sigma$-compact locally compact, 
then we can choose 
$\yorfmet\in \met(X)$
and 
$\yorvmet\in \met(\yorqsp{X}{F})$
so that 
$\yorqsp{X}{F}$ is compact and we have 
$\yorsubmapq(\yorpropset(F))\yosub
\yorpropset(X)$. 
In this case, we also have 
$\yorsubmapq(\yorpropsetu(F; S))
\yosub \yorpropsetu(X; S)$. 
\item\label{item::sep}
If $X$ is separable, then 
we can choose $v\in \met(\yorqsp{X}{F})$
such that
for all $\yorbd{a}\in \yorlset$ and 
$\yorbd{b}\in \yorrset$, we have 
$\yorsubmapq(\yordimset{F}{\yorbd{a}})
\yosub \yordimset{X}{\yorbd{a}}$
 and 
 $\yorsubmapq(\yordimsetq{F}{\yorbd{b}})
\yosub \yordimsetq{X}{\yorbd{b}}$. 
In this setting, we also notice the inclusions  
that
$\yorsubmapq(\yordimsetu{F}{S}{\yorbd{a}})
\yosub \yordimsetu{X}{S}{\yorbd{a}}$
 and 
 $\yorsubmapq(\yordimsetqu{F}{S}{\yorbd{b}})
\yosub \yordimsetqu{X}{S}{\yorbd{b}}$. 
\item\label{item::dense}
If we fix  $\eta \in (0, \infty)$, 
then we can choose $\yorvmet\in \met(\yorqsp{X}{F})$
so that for all $d\in \met(F)$, the set 
$F$ is $\eta$-dense  in 
$(X, \yorsubmapq(d))$. 
 
\end{enumerate}
\end{thm}

\begin{proof}
Theorem \ref{thm:subNARC} can be prove by 
a similar method to the proof of 
Theorem \ref{thm:subARC}. 
Recall Proposition \ref{prop:S-ult} and 
remark that 
to prove the statement \ref{item::cm},
we need to use  
\cite[Proposition 2.17]{Ishiki2021ultra} stating 
that $\yorcompsetu(Y; S)\neq \emptyset$
 if and only if 
$\yorcompset(Y)\neq \emptyset$ for every 
ultrametrizable space $Y$. 
\end{proof}

\begin{rmk}
Theorem \ref{thm:subNARC} contains 
the following author's theorems:
\begin{enumerate}
\item The extension theorem 
\cite[Theorem 1.2]{Ishiki2021ultra}
of ultrametrics. 
\item The simultaneous extension theorem \cite[Theorem 1.3]{ishiki2022highpower} of ultrametrics.
\item 

The extension theorem \cite[Theorem 3.15]{Ishiki2022proper} of proper ultrametrics.
\item 
The theorem
\cite[Theorem 4.9]{Ishiki2022proper} on the existence of an extended proper 
ultrametric which makes a given closed subset $\eta$-dense. 
\end{enumerate}
Moreover, the statements 
\ref{item:subARC:sep} and \ref{item::sep} 
can be considered as 
improvements of  \cite[Theorem 3]{stasyuk2009continuous}. 

\end{rmk}
\begin{rmk}
Theorems \ref{thm:main1} 
and \ref{thm:subNARC}
and 
the statements appearing in the proofs of them 
are still true for 
metrics taking values in 
general linearly ordered Abelian groups. 
This observation  gives another proof 
of \cite[Theorem 1.2]{ishiki2022highpower}. 
For the existence of retractions with the condition \ref{item:r22}, we can use \cite[Theorem 3.1]{ishiki2022highpower}, which is an analogue of
Brodskiy--Dydak--Higes--Mitra
\cite[Theorem 2.9]{brodskiy2007dimension}
(see Remark \ref{rmk:ret} in the present paper).  
\end{rmk}

\begin{rmk}
In Theorem \ref{thm:subARC}, 
the statements
\ref{item:subARC:cm}--\ref{item:subARC:dense}
except the pair of  \ref{item:subARC:cm} and 
\ref{item:subARC:sep}
 are compatible with each other.
For example, 
if we fix $\eta\in (0, \infty)$ and  $X$ is completely metrizable, then we can choose 
$\yorvmet\in \met(\yorqsp{X}{F})$ so that 
$\yorsubmap$   satisfies 
the conclusions of \ref{item:subARC:cm} and \ref{item:subARC:dense}. 
On the other hand, 
the author does not know whether
there exists an extensor 
$\yorsubmap$ satisfying  the conclusions of 
 \ref{item:subARC:cm} and \ref{item:subARC:sep}
  under the assumption that  
$X$ is separable and completely metrizable.
In general, we can not take a complete metric 
$v$ such that $v\in\yordimset{X}{\yorzeron}$. 
The existence of such a metric is equivalent to 
the $\sigma$-compactness and 
the local compactness of 
$\yorqsp{X}{F}$. 
(note that all doubling complete metric spaces are proper). 
Similarly, 
in Theorem \ref{thm:subNARC}, 
the statements 
\ref{item::cm}--\ref{item::dense}
except the pair of \ref{item::cm} and 
\ref{item::sep}
 are compatible with each other.
\end{rmk}

\begin{ac}
The author would like to thank the referee for 
helpful comments and suggestions. 
\end{ac}

\bibliographystyle{plain}
\bibliography{bibtex/fct.bib}

\end{document}